\documentclass{amsart}
\usepackage{amssymb}
\usepackage{amsfonts}
\usepackage{amssymb}
\usepackage{amsmath}
\usepackage{amsthm}
\usepackage{centernot}
\usepackage{enumerate}
\usepackage{tabularx}
\usepackage{centernot}
\usepackage{mathtools}
\usepackage{stmaryrd}
\usepackage{amsthm,amssymb}
\usepackage{etoolbox,color}
\usepackage{tikz}
\usepackage{amssymb}
\usetikzlibrary{matrix}
\usepackage{tikz-cd}
\usepackage{tikz}
\usepackage{marginnote}
\definecolor{mygray}{gray}{0.85}
\usepackage[linecolor=black, backgroundcolor=mygray,colorinlistoftodos,prependcaption,textsize=small]{todonotes}
\usepackage{xargs}  
\usepackage[colorlinks,citecolor=blue,urlcolor=blue, linkcolor=blue]{hyperref}

\renewcommand{\leq}{\leqslant}

\renewcommand{\geq}{\geqslant}

\newcommand{\mrm}[1]{\mathrm{#1}}

\usepackage[backgroundcolor=mygray,colorinlistoftodos,prependcaption,textsize=small]{todonotes}
\usepackage{xcolor}

\makeatletter
\def\subsection{\@startsection{subsection}{3}%
  \z@{.5\linespacing\@plus.7\linespacing}{.3\linespacing}%
  {\bfseries\centering}}
\makeatother

\makeatletter
\def\subsubsection{\@startsection{subsubsection}{3}%
  \z@{.5\linespacing\@plus.7\linespacing}{.3\linespacing}%
  {\centering}}
\makeatother

\makeatletter
\def\myfnt{\ifx\protect\@typeset@protect\expandafter\footnote\else\expandafter\@gobble\fi}
\makeatother

\newtheorem{theorem}{Theorem}[section]

\newtheorem{lemma}[theorem]{Lemma}
\newtheorem{proposition}[theorem]{Proposition}

\newtheorem{question}[theorem]{Question}

\theoremstyle{plain}

\newcommand{\leftarrowdbl}{\leftarrow\mathrel{\mkern-14mu}\leftarrow}
\newcommand{\preceqc}{\preccurlyeq}

\theoremstyle{definition}

\newtheorem{fact}[theorem]{Fact}
\newtheorem{definition}[theorem]{Definition}
\newtheorem{remark}[theorem]{Remark}

\newtheorem{observation}[theorem]{Observation}
\newtheorem{notation}[theorem]{Notation}

\newcounter{claimcounter}
\newenvironment{claim}{\stepcounter{claimcounter}{\noindent {\bf Claim \theclaimcounter.}}}{}
\newenvironment{claimproof}[1]{\noindent{{\em Proof of claim.}}\space#1}{\hfill $\rule{0.40em}{0.40em}$}

\begin{document}

\begin{abstract} 
 We prove that the epimorphism  relation   is a complete analytic quasi-order on the space of countable groups. 
  In the process, we obtain  the result of independent interest   that the epimorphism relation on pointed reflexive graphs is complete.
\end{abstract}

\title[The epimorphism relation  is a complete quasi-order]{The epimorphism relation  among    countable groups is a complete analytic quasi-order}

\thanks{S. Gao and F. Li acknowledge the partial support of their research by the Fundamental Research Funds for the Central Universities and by the National Natural Science Foundation of China (NSFC) grant 12271263. A. Nies was supported by the Marsden fund of New Zealand during the initial part of the project. 
G. Paolini was supported by project PRIN 2022 ``Models, sets and classifications", prot. 2022TECZJA and by INdAM Project 2024 (Consolidator grant) ``Groups, Crystals and Classifications''.}

\author{Su Gao}
\address{School of Mathematical Sciences and LPMC, Nankai University, Tianjin 300071, P.R. China}
\email{sgao@nankai.edu.cn}

\author{Feng Li}
\address{School of Mathematical Sciences and LPMC, Nankai University, Tianjin 300071, P.R. China}
\email{fengli@mail.nankai.edu.cn}

\author{Andr{\'e} Nies}
\address{School of Computer Science, The University of Auckland, New Zealand}
\email{andre@cs.auckland.ac.nz}

\author{Gianluca Paolini}
\address{Department of Mathematics ``Giuseppe Peano'', University of Torino, Via Carlo Alberto 10, 10123, Italy.}
\email{gianluca.paolini@unito.it}


\maketitle


\section{Introduction}

	The collection of groups with domain $\omega$ admits a natural topology that  makes it  a Borel space. Using the tools of invariant descriptive set theory (see~\cite{gao}), one can attempt to determine the (Borel) complexity of   equivalence relations and quasi-orders of interest on this space. Friedman and Stanley~\cite{friedman_and_stanley} proved  that the isomorphism relation on countable  groups is  Borel bireducible  with   the isomorphism  relation between   countable graphs, using a construction of   Mekler~\cite{mekler}.
	 In fact they show that the  isomorphism relation on the  class of $2$-nilpotent groups  of exponent a fixed odd prime is Borel bireducible with countable graph isomorphism. Recently, 
	 the fourth author of the present paper and Shelah~\cite{1205} proved  that the same applies to countable torsion-free abelian groups. The isomorphism relation between  finitely generated groups is a universal essentially countable Borel equivalence relation~\cite{veli}. In another direction,  Calderoni and Thomas~\cite{calderoni}  showed that the embeddability relation between countable torsion free abelian groups is a complete analytic quasi-order (and hence bi-embeddability is a complete analytic equivalence relation). Our aim is to determine the Borel complexity of a quasi-order on countable groups dual to that of embeddability:   being an epimorphic image.
	
	\smallskip
	
	Throughout, when we say a structure is countable we mean that it has   $\omega$ as its domain. Classes of structures we are interested in will  always  form a standard Borel space.  Camerlo~\cite{camerlo} showed that  the epimorphism relation on countable graphs is  a  complete   analytic quasi-order. Camerlo, Marcone and Motto Ros~\cite{camerlo2} then showed such a result for several other classes of structures. Notably,  the  case of   countable groups was   left open.
	The general complexity  of epimorphism relations on   Borel classes of countable structures is  considered in \cite[Question~6]{FMR}. 
	
	For countable groups $A,B$, we  write  $A \leftarrowdbl B$ if there is a surjective homomorphism from $B$ onto $A$; equivalently, $A$ is  isomorphic to a quotient of $B$. This relation   is easily shown to be  analytic. Our main result is that  it has  the maximum possible complexity:

	\begin{theorem}\label{th_homo_image} The relation   $\{\langle A,B \rangle  \colon A \leftarrowdbl B\}$  on the Borel space of countable   groups      is  a  complete   analytic quasi-order. Thus, the relation of bi-epimorphism on the same Borel space  is  a complete      analytic equivalence relation.
\end{theorem}

	Our reduction proceeds in   two steps.  In the first step, we use a result of Louveau and Rosendal \cite{LR} to show that the epimorphism relation  on countable    pointed reflexive graphs  is analytic complete; here a pointed reflexive graph is one that has a distinguished vertex that is adjacent to each vertex.  In the second step, we Borel reduce this quasi-order to the epimorphism relation  on countable groups, using 
	a new group theoretic construction based on certain countably generated Coxeter groups.

We leave open the following.

\begin{question}   Is      the   epimorphism relation  a complete analytic quasi-order on  the Borel space of countable {\em abelian} groups?
\end{question}

\section{Preliminaries}\label{sec_preliminaries}

	\begin{definition} 
	\begin{enumerate}[(1)]
	\item A graph is a structure $\Gamma = (V, R)$, where $V$ is a set and $R \subseteq V^2$ is a symmetric relation which is {\em irreflexive}, i.e., for every $a \in A$ we have that $(a, a) \notin R$.
	\item A {\em reflexive} graph is a structure $\Gamma = (V, R)$ such that $R$ is symmetric and $(a, a) \in R$ for every $a \in V$.
\end{enumerate}
\end{definition}

	\begin{definition} Let $\Gamma = (V_\Gamma, R_\Gamma)$ and $\Delta = (V_\Delta, R_\Delta)$ be graphs (resp. reflexive graphs). An homomorphism of graphs $f: \Gamma \rightarrow \Delta$ is a map $f$ from $V_\Gamma$ to $V_\Delta$ such that for every $a, b \in V_\Gamma$, if $(a,  b) \in R_\Gamma$, then $(f(a), f(b)) \in R_\Delta$.
\end{definition}

	\begin{definition}\label{def_R} Let $\Gamma$ be a (possibly empty) reflexive graph. The pointed reflexive graph $\Gamma_c$ based at $\Gamma$ is a reflexive graph with an extra distinguished vertex $c = c_\Gamma$ (named by a constant) such that $c$ is connected to every vertex of $\Gamma$. A pointed reflexive graph is a structure $\Gamma_c$ as above for some (possibly empty) reflexive graph~$\Gamma$.
\end{definition}

	\begin{remark} A morphism of pointed reflexive graphs $f:\Gamma_c \rightarrow \Delta_c$ is   a morphism of reflexive graphs   sending $c_\Gamma$ to $c_\Delta$. Suppose that there is a subset $A$ of $\Gamma$ and a morphism $g: A \rightarrow \Delta$, where by $A$ (considered as a graph) we denote the induced subgraph of $\Gamma$ on vertex set $A$. Notice that, since the graph $\Gamma_c$ is reflexive and $c_\Gamma$ is connected to every element of $\Gamma$, we can define a morphism of pointed reflexive graphs $f: \Gamma_c \rightarrow \Delta_c$ by sending any $b \in \Gamma \setminus A$ to $c_\Delta$, and letting $f(a) = g(a)$ for any $a \in A$.
\end{remark}

	\begin{definition} Let $A, B$ be structures for the same signature.   An \textit{epimorphism} $f: A \twoheadrightarrow B$ is a morphism that  is surjective.
\end{definition}

\section{The epimorphism relation between pointed reflexive graphs is analytic complete}

\begin{definition}
	For   countable graphs $\Gamma$ and $\Delta$, write  $\Gamma \preceqc_1 \Delta$ if there is an \textit{injective} homomorphism from $\Gamma$ to $\Delta$.  \end{definition}

\begin{remark} \label{LRrem}
Louveau and Rosendal  \cite[Theorem 3.5]{LR}     proved that the homomorphism relation between countable graphs is a complete analytic quasi-order. They  give a Borel reduction $T \to G_T$ of their complete quasi-order  $\leq_{max} $  on normal trees (see  \cite[Definition  2.3 and above]{LR}) to it.  One observes from the proof that for each normal tree $T$, the constructed graph $G_T$ has no isolated vertices. As they remark at the end of their proof, if  $S \le_{max} T$ then $G_S$ is in fact isomorphic to an induced subgraph of $G_T$, and so in particular $G_S\preceqc_1 G_T$.  Thus, their proof    shows that \emph{the relation $\preceqc_1$ on   countable graphs with no isolated vertices is a complete analytic quasi-order.}
\end{remark}

  Recall   the definition of a pointed reflexive graph from Definition~\ref{def_R}.

 \begin{definition}	For countable reflexive graphs $\Gamma$ and $\Delta$, we write  $\Gamma\leftarrowdbl \Delta$ if there is an onto homomorphism from $\Delta$ to $\Gamma$. 
\end{definition}

\begin{proposition}\label{reduction}
	There is a Borel map $F$ from the class of countable   graphs without isolated vertices to the class of countable pointed reflexive graphs, such that for any countable connected graphs $\Gamma$ and $\Delta$, we have
	$$\Gamma \preceqc_1 \Delta \iff F(\Gamma) \leftarrowdbl F(\Delta). $$
	Thus, the relation $\leftarrowdbl$ between  countable pointed reflexive graphs is a complete analytic quasi-order using  Remark~\ref{LRrem}.
\end{proposition}

\begin{proof}
	If $\Gamma=(V,R_\Gamma)$ is a countable    graph without isolated vertices, we let  $F(\Gamma)= \Delta_c$ according to Definition~\ref{def_R}, where $\Delta $ is the graph   $(V, V^2-R_\Gamma)$ and $c = c_{F(\Gamma)}$ is a fresh vertex. Thus, 
	$$F(\Gamma)=\left(V\cup \{c\}, (V^2\setminus R_\Gamma) \cup (\{c\}\times V)\cup (V\times \{c\})\cup \{(c,c)\} \right). $$
  Then $F(\Gamma)$ is a pointed reflexive graph (since $\Gamma$ is irreflexive by assumption) with $c$ as the distinguished vertex. Since $\Gamma$ has no isolated vertices,  $c$ is the only vertex in $F(\Gamma)$ which is adjacent to every vertex of $F(\Gamma)$. It is easily seen that the map $F$ is Borel. We prove that $F$ is a reduction from $\preceqc$ to $\leftarrowdbl$. 
	
	\smallskip \noindent   Suppose first that $f\colon \Gamma\to \Delta$ is an injective homomorphism. We construct an epimorphism $g$ from $F(\Delta)$ to $F(\Gamma)$ as follows. Fix $y\in F(\Delta)$. If $y\notin f(\Gamma)$, we set $g(y)= c_{F(\Gamma)}$. If $y=f(x)$ for some $x\in \Gamma$, then we set $g(y)=x$. Since $f$ is injective, $g$ is well defined. Clearly $g$ is an onto map, and $g(c_{F(\Delta)})=c_{F(\Gamma)}$. To see that $g$ is a homomorphism, consider $y,y'\in F(\Delta)$. Since both $F(\Gamma)$ and $F(\Delta)$ are reflexive graphs we may  assume without loss of generality that $y \neq y'$. 
	If either of $y, y'$  is not in the range of $f$, then its image under $g$ will be $c_{F(\Gamma)}$, which is adjacent to every vertex in $F(\Gamma)$ by definition. Thus we may assume $y=f(x)$ and $y'=f(x')$ for  distinct vertices $x , x'$ of $\Gamma$. 
	Then $g(y)=x$ and $g(y')=x'$.  Suppose  that $y$ and $ y'$ are adjacent in $F(\Delta)$. If $x$ and $x'$ are not adjacent in $F(\Gamma)$, then by   definition  they are adjacent in $\Gamma$. Thus, since $f$ is an homomorphism, $y, y'$ are adjacent in $\Delta$,  contradiction. Thus $g$ is a homomorphism as required.

	\smallskip \noindent    Conversely, suppose that $g\colon F(\Delta)\to F(\Gamma)$ is an epimorphism. We construct an injective homomorphism from $\Gamma$ to $\Delta$. Before doing this, we note that, since $c_{F(\Gamma)}$ is the only vertex in $F(\Gamma)$ adjacent to every vertex in $F(\Gamma)$, and similarly for $F(\Delta)$, we must have $g(c_{F(\Delta)})=c_{F(\Gamma)}$. Now fix a vertex $x$ of $\Gamma$. Since $g$ is onto and $g(c_{F(\Delta)})=c_{F(\Gamma)}$, there is a vertex $y$ of $\Delta$ such that $g(y)=x$. Choose as  $f(x)$     any such $y$. Then $f$ is injective. We verify  that $f$ is a homomorphism. Let $x,x'$ be adjacent vertices of $ \Gamma$ (so necessarily $x \neq x'$, since $\Gamma$ is  irreflexive).  If    $f(x)$ is not adjacent to $f(x')$ in $\Delta$, then they are adjacent in $F (\Delta)$. Since $g$ is a homomorphism this means $x=g(f(x))$ is adjacent to $x'= g(f(x'))$ in $F(\Gamma)$, contradiction. 
\end{proof}

\section{The group theoretic construction}

\begin{definition}
A \textit{Coxeter graph} $\Delta = (S, R) $  is  a reflexive, undirected, \textit{weighted}   graph with edges $\{s, t\}$   labeled by numbers $m(s, t) \in  \mathbb N \setminus \{0\}$, with $m(s,t) = m(t,s)$ for each $s,t$, and  $m(s,t) = 1 $ iff $s=t$. To   such a  graph one associates a group $G= G(\Delta)$, called the \textit{Coxeter group of type}~$\Delta$. It is given by the presentation
\[ \langle S  \mid   (st)^{m(s,t)} = e \text{ for } s, t \in S, \;sRt  \rangle. \]
One says that $(G, S)$ is  a \textit{Coxeter system}  of type $\Delta$.
\end{definition}  
 Note that each $s \in S$ is an involution in $G(\Delta)$, and that $m(s,t)=2$ means that $s$ and $t$ commute. 
 \begin{definition}
 	Let $(G, S)$ be a Coxeter system.  A  finite subgroup $U$  of $G$ is called \textit{special} if  it is generated by $U \cap S$. Conjugates of special subgroups are called \textit{parabolic}. 
 \end{definition}

\begin{fact}[{\cite[Proposition~2.87]{abramenko}; cf. \cite[Theorem 12.3.4]{davis}}]\label{finite_subgroups_fact} Let $(G, S)$ be a Coxeter system. Then any finite subgroup $A$  of $G$ is contained in a finite parabolic subgroup of $G$. 
\end{fact}

Let  $L_4$ (usually denoted $A_4$ in the literature) be the weighted graph on four vertices $s_1, ..., s_4$ such that, for $i , j \in \{1, 2, 3, 4\}$, $m(s_i, s_{j}) = 3$ if $|j-i|=1$, and $m(s_i, s_{j}) = 2$ if  $|j - i|> 1$.  The standard presentation of $S_n$, the symmetric group of $n$ elements, with   transpositions of adjacent elements  as generators yields:
	\begin{observation}\label{S_4_obs} The symmetric group on $5$ elements $S_5$ is isomorphic to $G(L_4)$. 
\end{observation}

	\begin{definition}\label{the_crucial_def} \label{Alt_notation}  Given a graph $\Gamma=(S,E)$  we define a  Coxeter graph $C_\Gamma$ as follows. 
	\begin{enumerate}[$\bullet$]
	 \item We replace each $s\in S$ by vertices  $s_1, ..., s_4$. \item We add labelled edges between them to ensure that the subgraph induced on   $ s_1, \ldots, s_4$ is $L_4$.
	 \item For distinct  $s,t \in S$,  we add an edge labelled~$2$  between $s_4$ and $t_4 $.
     \item For distinct $s,t \in S$ such that $s$ and $t$ are adjacent in $\Gamma$ and  $i= 1,3$,  we add edges labelled~$2$  between $s_i$ and $ t_i $ and $s_i $ and $t_{4-i}$.  
	 \end{enumerate}
		
\noindent	By $S_5(s)$ we denote the subgroup $\langle s_1, ..., s_4 \rangle_{G(C_\Gamma)} \cong G(L_4) \cong S_5$ and by $A_5(s)$ we denote the alternating subgroup of  index $2$ in  $S_5(s)$.
	 \end{definition}

	\begin{notation} For $(G, S)$ a Coxeter system and $J \subseteq S$, we let $G_J = \langle J \rangle_G$.
\end{notation}

\begin{remark}\label{remark_commuting_involutions}
 Given a graph $\Gamma=(S,E)$  for each   $s\in S $  we have   $s_1s_3\in A_5(s)$, and $s_1s_3$  is an involution. Furthermore, if $s E t$, then $G(C_\Gamma) \models [s_1 s_3, t_1t_3  ]=e$.
\end{remark}

	\begin{lemma}\label{char_of_the_Alt_5} Let $G = G(C_\Gamma)$ be as in Definition~\ref{the_crucial_def}. 
\begin{enumerate}[(1)]
	\item Each subgroup $A \cong A_5$ of $G$ is conjugate to $ A_5 (s) $ for some  vertex $s$ of $ \Gamma$.
	\item Each subgroup $B \cong S_5$ of $G$ is conjugate to $S_5 (s) $ for some  vertex $s$ of $ \Gamma$.
	\end{enumerate}
\end{lemma}

	\begin{proof} Item (2) follows from (1). We prove (1).	 By Fact~\ref{finite_subgroups_fact},   $A$ is contained in a finite parabolic subgroup of $G$.  We analyse what the corresponding special subgroup can be. Let $V$ denote the vertex set of   $C_\Gamma $.  First note that if $u, v\in V$ and there is no edge between $u$ and $v$ in $C_\Gamma$, then $\langle u, v\rangle_G$ is isomorphic to $\mathbb{Z}/2\mathbb{Z}*\mathbb{Z}/2\mathbb{Z}$, which is infinite. Let $T\subseteq V$ be a finite set generating a special group which contains a conjugate of $A$. Then $G_T=\langle T\rangle_G$ is finite and contains at least 60 elements. An analysis of the situation gives that $T$ must be contained 
	in a set of the form $\{s_1, s_3 : s \in K\}$ or $\{s_4 : s \in K \}$ with $K$ a clique of $\Gamma$, or $T = \{s_1, ..., s_4\}$, for some $s \in S$.
	In the first two  cases  $G_T$  is abelian.  In the latter  case $G_T \cong  S_5$.  So
		  only in the third case the  special  subgroups of $G$  can  contain copies of $A_5$.  Furthermore, any such subgroup  contains exactly one copy of $A_5$. So by Fact~\ref{finite_subgroups_fact},   $A $ is conjugate to  $A_5(s)$ for some $s \in \Gamma$, as required.
\end{proof}

	In order to prove Lemma~\ref{non_conjugate_lemma} we need some  facts from the theory of Coxeter groups, so from Definition~\ref{def_op} to Fact~\ref{krammer_fact} we introduce what is necessary in order to prove Lemma~\ref{non_conjugate_lemma}.
	
	\begin{definition}\label{def_op} Given a Coxeter graph $\Gamma$ we denote by $\Gamma^{\mrm{op}}$ the graph which replaces the edges labelled by $2$ of $\Gamma$ with non-edges and the non-edges of $\Gamma$ with edges labelled by $\infty$.
\end{definition}


	\begin{definition}[{\cite[Definition 4.5.1]{davis}}]\label{def_PI} Let $(G, S)$ be a Coxeter system and let $T\subseteq S$. The length of an element $w\in G$ is the length of a shortest expression of $w$ as a product of elements of $S$. An element $w\in G$ is $(T, \emptyset)$-reduced if it is an element of minimum length in the coset $G_Tw$. We define the fundamental $T$-sector in $G$ by 
$$\Pi_T = A_T := \{w \in G : w \text{ is } (T, \emptyset)\text{-reduced}\}.$$
\end{definition}

By \cite[Lemma 4.3.1]{davis}, each right coset of $G_T$ contains a unique $(T,\emptyset)$-reduced element. Thus the fundamental $T$-sector in $G$ is a set of representatives for the right cosets of $G_T$.

\begin{definition}\label{def_long} Let $(G, S)$ be a Coxeter system. If $I \subseteq S$ and $G_I$ is finite, then $G_I$ has a unique element of maximum length (\cite[Lemma 4.6.1]{davis}). 
This element is called the longest element of $(G_I , I)$ and it is denoted by $w_I$.
\end{definition}
	
	\begin{definition}[{\cite[Definition 4.10.3]{davis}}]\label{def_K} Let $(G, S)$ be a Coxeter system of type $\Gamma$.
Let $I \subseteq S$, $s \in S \setminus I$. Write $K$ for the connected component of $I \cup \{s\}$ containing $s$ in the graph $\Gamma^{\mrm{op}}$. If $G_K$ is finite, 
then we define
\[
\nu(I, s) = w_{K \setminus \{s\}} w_K.
\]
Otherwise, $\nu(I, s)$ is not defined. Whenever $\nu(I, s)$ is defined, we have
\[
\nu(I, s)^{-1} \Pi_I = \Pi_J
\]
for some $J = (I \cup \{s\}) \setminus \{t\}$, $t \in K$ (\cite[Lemma 4.10.4]{davis}). We then define the directed labelled graph $\mathcal{K}_S$ whose vertices are the subsets of $S$ and in which there is an edge from $I$ to $J$, labelled $s$, each time $\nu(I, s)$ exists, and $\nu(I, s)^{-1} \Pi_I = \Pi_J$.
 \end{definition}

\begin{fact}[{\cite[Theorem 4.1.6 (i)]{davis}}; cf. {\cite[2.4.1(i)]{brenti}}] 
Let $(G, S)$ be a Coxeter system and let $J \subseteq S$. Then $(G_J, J)$ is a Coxeter system.
\end{fact} 
	
	\begin{fact}[{\cite[Corollary 3.1.7]{krammer}}]\label{krammer_fact} Let $(G, S)$ be a Coxeter system and let $I, J \subseteq S$. Then $G_I$ and $G_J$ are conjugate in $G$ if and only if
$I$ and $J$ are in the same connected component of $\mathcal{K}_S$, where $\mathcal{K}_S$ is as defined in Definition~\ref{def_K}.
\end{fact}

	\begin{lemma}\label{non_conjugate_lemma} Let $s \neq t \in \Gamma$. Then   $S_5(s)$ is not conjugate to    $S_5(t)$  in $G(C_{  \Gamma})$.
\end{lemma}
  
	\begin{proof} Let $(G, X) = (G(C_{  \Gamma}), C_{  \Gamma})$. Obviously $S_5(s) = G_I$ for $I = \{s_1, ..., s_4\}$ and $S_5(t) = G_J$ and $J = \{t_1, ..., t_4\}$, so the result follows from Fact~\ref{krammer_fact} by simply looking at the definition of $\mathcal{K}_S$ from Definition~\ref{def_K} in this case. Specifically, for every $v \in \Gamma$, letting $T = \{v_1, ..., v_4 \}$ we have that every $x \in X \setminus T$ is such that the connected component (in the sense of $C(\Gamma)^{\mrm{op}}$) of $T \cup \{x\}$ containing $x$ is always $T \cup \{x\}$ itself, since $x$ is connected to $T$ by an $\infty$-edge in $C(\Gamma)^{\mrm{op}}$; furthermore, $G_{T \cup \{x\}}$ is always infinite (since it is connected by an $\infty$-edge). So $\mathcal{K}_S$ does not contain edges that start at $T$ in the sense of Definition~\ref{def_K}, since that requires $\nu(T, s)$ to exist and in particular that the connected component $K$ is finite, which never happens in our case.
\end{proof}


	\begin{lemma}\label{lemma_for_surhection1} Let $\Gamma$ and $\Delta$ be reflexive graphs.  Suppose that  $\varphi \colon G(C_{  \Gamma}) \to  G(C_{  \Delta})$ is an onto group homomorphism. Then for every vertex $v $ of $ \Delta$, there is  a vertex $s $   of $ \Gamma$ such that $\varphi (S_5(s)) $ is conjugate to $S_5(v)$ in $ G(C_{  \Delta})$. 
\end{lemma}

	\begin{proof}  
		Assume for a contradiction that  there is a vertex $w$ of $ \Delta$ 
		  such that no subgroup of the form $\varphi(\mrm{S_5}(v))$, for a vertex $v $ of $ \Gamma$, is   conjugate of $\mrm{S_5}(w)$. 
		  Let $\pi$ be the homomorphism of $G(C_\Delta)$ onto $S_5(w)$ that sends each generator $v_i$, for a vertex $v$ of $\Delta$, distinct to $w$, to $e$, and $w_i$ to $w_i$.
		 Note that  $\eta := \pi  \circ \varphi$  is an onto map 
		 $$G(C_{  \Delta}) \to S_5(w).$$
		
\smallskip \noindent 	For any  vertex $v \neq w$ of $\Delta$, the map  $\pi$ sends any conjugate of $S_5(v)$  to $\{e\}$. Thus, by the hypothesis on $\varphi$,  for each vertex $s$ of $\Gamma$, the map $\eta \restriction S_5(s)$ is not 1-1. Since the only proper quotients of $S_5$ are $\mathbb{Z}/2\mathbb{Z}$ and the trivial group, this implies $\eta (s_i)= \eta(s_4)$ for each vertex $s$ of $\Gamma$. Since $[r_4, s_4]= 1$ in  $G(C_{  \Gamma})$ and $\eta$ is onto, we conclude that   $\{ \eta(s_4) : s \text{ a vertex of } \Gamma\}$ form a generating set of pairwise commuting elements for $S_5$, contrary to the fact that $S_5$ is nonabelian. 
\end{proof}

	Lemma~\ref{lemma_Feng} below will be crucially needed in the proof of Claim~2 from the proof of Theorem~\ref{the_hammer}. In order to prove it we also need to state a few facts and definitions.

We first recall (see \cite[Section 4.2]{brenti}) that every Coxeter system $(G, S)$ admits a geometric representation $\sigma\colon G\to GL(E)$ where $E$ is a real vector space of dimension $|S|$. More specifically, let $\{\alpha_s\}_{s\in S}$ be a basis of $E$. One can define, for each $s\in S$, a linear mapping $\sigma_s\colon E\to E$ by a bilinear form associated to the Coxeter system (see \cite[Section 4.2]{brenti} for details). Then $\sigma\colon G\to GL(E)$ can be obtained as a unique group homomorphism extending $s\mapsto \sigma_s$ (\cite[Theorem 4.2.2]{brenti}). Under this representation, the root system of $(G, S)$ is $\{w(\alpha_s)\}_{w\in G, s\in S}$; its elements are called roots. The elements of $\{\alpha_s\}_{s\in S}$ are the simple roots. In the following, we fix a geometric representation for each Coxeter system.

	\begin{definition}[\cite{rich}] We say that a Coxeter system $(G, S)$ is of $(-1)$-type if $G$ is finite and the action of the longest element $w = w_S$ of $(G, S)$ (cf. Definition~\ref{def_long}) satisfies that $w \cdot \alpha_s = -\alpha_s$ for any $s \in S$, where $\alpha_s$ is the simple root associated to $s \in S$.
\end{definition}	
	
	\begin{fact}[{\cite[Theorem~A]{rich}}]\label{fact1_Feng}
Let $(G, S)$ be a Coxeter system and $w$ be an element of $G$ with $w^2 = 1$. Then there exist a subset $I \subseteq S$ and an element $u \in G$ such that $(G_I, I)$ is of $(-1)$-type and $uwu^{-1} = w_I$.
\end{fact}

	Given a group $G$, an element $w \in G$ and a subgroup $H \leq G$ we denote by $Z_G(w)$ the centralizer of $w$ in $G$ and by $N_G(H)$ the normalizer of $H$ in $G$.

\begin{fact}[{\cite[Proposition~2.16 (i)]{nuida}}]\label{fact2_Feng}
Let $(G, S)$ be a Coxeter system and $I$ be a subset of $S$ such that that $(G_I, I)$ is of $(-1)$-type. Then $Z_G(w_I) = N_G(G_I)$.
\end{fact}

	\begin{definition} Let $(G, S)$ be a Coxeter system and $I$ be a subset of $S$. Let 
	$$H_I = \{g \in G : g \Pi_I = \Pi_I\},$$
where $\Pi_I$ is as in Definition~\ref{def_PI}.
\end{definition}

	The following fact was originally proved in \cite{brink}, but see also \cite[Proposition~3.1.9 (a)]{krammer}. In fact we will use \cite[Corollary~3.1.5]{krammer} in the proof of Lemma~\ref{lemma_Feng}.

\begin{fact}[{\cite[Proposition 3.1.9 (a)]{krammer}}]\label{fact3_Feng}
Let $(G, S)$ be a Coxeter system and $I$ be a subset of $S$. Then 
$$N_G(G_I) = G_I \rtimes H_I.$$
\end{fact}

\begin{fact}[{\cite[Corollary 3.1.5]{krammer}}]\label{factHI} Let $(G, S)$ be a Coxeter system and $I$ be a subset of $S$. Let $\mathcal{K}^0$ be the connected component of $\mathcal{K}_S$ containing $I$. Let $\mathcal{T}$ be a spanning tree of $\mathcal{K}^0$. For each $J\in \mathcal{T}$, define
$$ \mu(J)=\nu(I_0, s_0)\dots \nu(I_t, s_t) $$
where 
$$ I=I_0\stackrel{s_0}{\longrightarrow}I_1\stackrel{s_1}{\longrightarrow}\cdots\stackrel{s_t}{\longrightarrow}I_{t+1}=J $$
is the unique non-reversing path in $\mathcal{T}$ from $I$ to $J$. Then $H_I$ is generated by all elements of the form
$$ \mu(I_0)\nu(I_0,s_0)\mu(I_1)^{-1} $$
where $I_0\stackrel{s_0}{\longrightarrow}I_1$ is an edge of $\mathcal{K}^0$ which is not an edge of $\mathcal{T}$.
\end{fact}

	\begin{lemma}\label{lemma_Feng} Let $\Gamma$ be the graph with two points $s$ and $t$ not connected by an edge and let $G = G(C_\Gamma)$. Also, let $a \in A_5(s)$ be an involution. Then $Z_G(a) \subseteq S_5(s)$.
\end{lemma}

	\begin{proof} Let $(G, X) = (G(C_{  \Gamma}), C_{  \Gamma})$, so that $X = S \cup T$, where $S = \{s_1, ..., s_4\}$ and $T = \{t_1, ..., t_4\}$. By Facts~\ref{fact1_Feng} and \ref{fact2_Feng} it suffices to show that if $I \subseteq S$ is of size $\geq 2$, then the subgroup $N_G(G_I)$ from Fact~\ref{fact3_Feng} is a subgroup of $S_5(s)$. But using Fact~\ref{factHI} and arguing as we did in Lemma~\ref{non_conjugate_lemma} it is easy to see that, for this choice of $I$, $N_G(G_I)$ is a subgroup of $S_5(s)$ (since outside of $S_5(s)$ $\infty$-edges are involved for such an~$I$).
	\end{proof}

	Recall the definition of pointed reflexive graphs from Definition~\ref{def_R}.

\begin{theorem}\label{the_hammer} Let $\Gamma_{c}$ and $\Delta_{c}$ be pointed reflexive graphs. 
	There is a     homomorphism from $\Gamma_c$ onto $\Delta_c$ if and only if there is a   homomorphism from $G(C_\Gamma)$ onto $G(C_\Delta)$.
\end{theorem}

	\begin{proof} 
 	For a pointed reflexive graph $\Gamma_c$, we let  $L(\Gamma) = G(C_{  \Gamma})		$, namely the Coxeter group given by the Coxeter graph associated with $\Gamma$.
 	
 	\smallskip \noindent 
	 Suppose $f: \Gamma_c \twoheadrightarrow  \Delta_c$ is an onto homomorphism of reflexive graphs. We  define the onto homomorphism  $L(f): L(\Gamma) \twoheadrightarrow L(\Delta)$.  If $f(s) = c_\Delta$, then, for $i = 1, ..., 4$,  send $s_i$  to $e$. Otherwise,   send  $s_i$   to $f(s)_i$.  The following shows that this assignment on the generators of  $L(\Gamma)$   describes  a homomorphism of groups (which is clearly onto). 
		 
\smallskip \noindent 	
		 
 		 \begin{claim} Each relator in the presentation of $G(C_\Gamma)$ is sent either to $e$ or  to a  relator in the presentation of $G(C_\Delta)$.  \end{claim}

\begin{claimproof}	
%
%
Firstly,  assume that   the relator is of type $(ab)^2$. If $L(f)$ sends one of $a,b$ to $e$, the image of the relator is  in the presentation of $G(C_\Delta)$ because it stipulates that generators are involutions.  Otherwise the image of the relator is there because $f$ preserves edges.  

\smallskip \noindent 
Secondly, assume the relator  is of form $(ab)^3$.  W.l.o.g.\ we have $a=s_i, b= s_{i+1}$ for some vertex $s$ of $\Gamma$ and some $ i$. Then either both $a,b$ and hence the relator are sent to $e$; or they are sent to $t_i, t_{i+1}$ where $f(s)= t$, in which case  the relator  is  in the presentation of $G(C_\Delta)$.  This shows the claim. 
		\end{claimproof}
		
\smallskip \noindent 
	For an object $G=L(\Gamma) $, let $K(G)$ be the graph with vertex set the conjugacy classes $[A]$  of subgroups $A$ of $G$ that  are isomorphic to $S_5$. There is an edge between  two vertices $[A], [B] \in K(G)$ if there are $C \in [A]$ and $D \in [B]$ and involutions  $c\in \mrm{Alt}(C)$ and $d\in \mrm{Alt}(D)$ such that $c$ commutes with $d$, where  $\mrm{Alt}(C)$ denotes the unique subgroup of $C$ isomorphic to~$A_5$.

\smallskip

		\begin{claim}   $\Gamma \cong K(G(C_\Gamma))  $ for each pointed reflexive graph $\Gamma_c$.	\end{claim}

			\begin{claimproof} Define a graph  isomorphism $\alpha \colon \Gamma \to K(G(C_\Gamma))$. Given a vertex  $v $ of $ \Gamma$ we define $\alpha(v) $ to be the conjugacy class of  the subgroup $\mrm{S}_5(v)$. By Lemma~\ref{non_conjugate_lemma},  $\alpha$  is injective,  and   by Lemma~\ref{char_of_the_Alt_5}, $\alpha$ is   surjective.  
				
\smallskip \noindent					It remains to show  $\alpha$ is a graph isomorphism.   Write $G= G(C_\Gamma)$.  Clearly adjacent vertices  in  $\Gamma$ are mapped to adjacent vertices  in $K(G)$. Assume for a contradiction  that distinct vertices  $s,t$  of $ \Gamma$ are not adjacent  but  their images are adjacent. Then  there are   $ A \in [A_5(s)]$ and  $B \in [A_5(t)]$ such that some involution $a\in A$ commutes with an involution  $b\in B$.  There is a retraction $\rho: G \rightarrow \langle L_4(s) \cup L_4(t) \rangle_G =: H$ (recall Definition~\ref{Alt_notation}) which sends all the generators in $C_\Gamma \setminus (L_4(s) \cup L_4(t))$ to $e$.    Within $H$,     $\rho(a)$ is in  a conjugate of $A_5(s)$ and $\rho(b)$ is in a conjugate of $A_5(t)$.   Modulo an inner automorphism we have that $\rho(a) \in A_5(s)$ but this leads to a contradiction, as under our assumption $\rho(a) $ commutes with $\rho(b)$, which cannot be the case by Lemma~\ref{lemma_Feng}.
		\end{claimproof}

\smallskip \noindent Now, let $K_c(G(C_\Gamma)) \cong \Gamma_c$ be the pointed graph obtained from the graph $K(G(C_\Gamma))$ by adding an extra element $c_\Gamma$ to $K(G(C_\Gamma))$ connected to every element of $K(G(C_\Gamma))$, and similarly for $\Delta$, so that $K_c(G(C_\Delta)) \cong \Delta_c$ with $c_\Delta$ added to $K(G(C_\Delta))$. 

\smallskip \noindent Next we define $$K(\varphi): K_c(G(C_\Gamma)) \rightarrow K_c(G(C_\Delta))$$
for  an onto homomorphism $\varphi: G(C_\Gamma) \twoheadrightarrow G(C_\Delta)$. Recall that the group  $A_5$ is  simple. So   $\varphi$ sends each subgroup  $A$ of $G(C_\Gamma)$   isomorphic to $S_5$  either to  a group of size $\leq 2$, or to a subgroup of $G(C_\Delta)$  isomorphic to $S_5$.   Given $x= [A]$, let   $K(\varphi)(x) = c_\Delta$ in the former case, and $[\varphi(A)] $ in the latter case. 

\smallskip \noindent
\begin{claim} $\Phi := K(\varphi) \colon K_c(G(C_\Gamma) ) \to K_c(G(C_\Delta) ) $ is an epimorphism of pointed reflexive graphs.
\end{claim}

	\begin{claimproof} $\Phi$  is surjective  by Lemma~\ref{lemma_for_surhection1}.	We show $\Phi$ preserves edges. Let $[A],[B]$ be vertices of $K(G(C_\Gamma) )$ and suppose that $\Phi([A]) \neq c_\Delta$ and $\Phi([B]) \neq c_\Delta$ (otherwise there is nothing to prove), and suppose   there is an edge between $[A] $ and $[B]$. By the foregoing claim, we may assume that $A = S_5(s)$, $B= S_5(t)$ where $\{s,t\}$   is an edge of $\Gamma$. Then $\hat A= \varphi(A)$ and $\hat B= \varphi(B)$ are  subgroups of $G(C_\Delta)$ isomorphic to $S_5$ such that there are two distinct involutions $a_1, a_3 \in \hat A$ and two distinct involutions $b_1, b_3 \in \hat B$ such that $\langle a_1, a_3, b_1, b_3 \rangle_{L(\Delta)}$ is abelian. By the foregoing claim, this means that  there is an edge in $K(G(C_\Delta) )$ between $[\hat A]$ and $[\hat B]$. 
	\end{claimproof}

\smallskip
\noindent This concludes the proof, since $\Gamma_c \cong K_c(G(C_\Gamma) )$ and $K_c(G(C_\Delta) ) \cong \Delta_c$.
\end{proof}

 We are finally in the position to obtain the main result:
 \begin{proof}[Proof of Theorem~\ref{th_homo_image}] Clearly  the map $\Gamma_c \mapsto \Gamma(C_\Gamma)$ is a Borel map from the space of countable pointed reflexive graphs to the space of  countable groups. The result now  follows from    Proposition~\ref{reduction} and  Theorem~\ref{the_hammer}.
 \end{proof}

\end{document}